\documentclass[10pt]{article}
\usepackage{latexsym}
\usepackage{amssymb,amsbsy,amsmath,amsfonts,amssymb,amscd,amsthm}
\usepackage{epsfig, graphicx}
\usepackage{color}
\usepackage{enumerate}
\usepackage{pdfsync}
\newcommand{\R}{\mathbb{R}}

\newcommand{\Sph}{\mathbb{S}}

\newcommand {\sgn} { {\rm sgn} }

\newcommand {\caL} { {\cal L} }
\newcommand {\cau}  {{\cal U}}
\newcommand {\cax}  {{\cal X}}

\newcommand {\caw} { {\cal W} }

\newcommand\Approx{\mathop{\approx}}

\newcommand {\f}   {\frac}

\newcommand{\beq}{\begin{equation}}
\newcommand{\beqa}{\begin{eqnarray}}
\newcommand{\bea} {\begin{array}{ll}}
\newcommand{\beqan}{\begin{eqnarray*}}
\newcommand{\eeq}{\end{equation}}
\newcommand{\eeqa}{\end{eqnarray}}
\newcommand{\eeqan}{\end{eqnarray*}}
\newcommand{\eea} {\end{array}}
\newtheorem{theorem}{Theorem}[section]
\newtheorem{lemma}[theorem]{Lemma}

\newtheorem{proposition}[theorem]{Proposition}
\newtheorem{corollary}[theorem]{Corollary}
\newtheoremstyle{remarkb}
	{}
	{}
	{\normalfont}
	{}
	{\bfseries}
	{}
	{ }
	{}
\theoremstyle{remarkb}
\newtheorem{remark}[theorem]{Remark}
\numberwithin{equation}{section}
\newcommand{\cqfd}{{ \hfill
                       {\unskip\kern 6pt\penalty 500
                       \raise -2pt\hbox{\vrule\vbox to 6pt{\hrule width 6pt
                       \vfill\hrule}\vrule} \par}   }}
\title{\Large \bf Blow-up dynamics of self-attracting diffusive particles driven by competing convexities}

\author{Vincent Calvez~$^{\text{a}}$, Lucilla Corrias~$^{\text{b}}$
}
\date{\today}
\begin{document}
\maketitle
\pagestyle{plain}
\pagenumbering{arabic}

\begin{abstract}
In this paper, we analyze the dynamics of an $N$ particles system evolving according the gradient flow of an energy functional. The particle system is a consistent  approximation of the Lagrangian formulation of a one parameter family of non-local drift-diffusion equations in one spatial dimension. We shall prove the global in time existence of the trajectories of the particles (under a sufficient condition on the initial distribution) and give two blow-up criteria. All these results are consequences of the competition between the discrete entropy and the discrete interaction energy. They are also consistent with the continuous setting, that in turn is a one dimension reformulation of the parabolic-elliptic Keller-Segel in high dimensions. 
\end{abstract}

{\bf Key words.} Chemotaxis, blow-up, particles methods, gradient flow.
\bigskip

{\bf AMS subject classification:} 35B44; 35D30; 35Q92; 35K55; 65M99; 92C17; 92B05.

\section{Introduction}
\label{Sec:intro}

This paper aims to give an insight of the blow-up dynamics driven by the competitions between random and directed movements undergone by particle systems. More specifically, we consider the following one parameter family of one dimensional non-local drift-diffusion equations describing self-attracting diffusive particles
\begin{equation}
\partial_t\rho=\partial_{xx}\,\rho - \partial_x(\rho\, \partial_x S_\gamma(\rho))\,,\quad x\in\R\,,\ t>0\,.
\label{eq:gamma_eq}
\end{equation}
Here $\rho$ is the density of the particles and $S_\gamma(\rho)=K_\gamma*\rho$ is the interaction potential,  with the interaction potential kernel $K_\gamma$ given by
\beq
K_\gamma (x):=\gamma^{-1}\,|x|^{-\gamma}\,,\qquad \gamma\in(0,1)\,.
\label{eq:Kgamma}
\eeq

Non-local equations of type \eqref{eq:gamma_eq} in space dimension $d\ge1$, have been introduced in several domains. Let us recall the parabolic-elliptic Keller-Segel system \cite{KS1,KS2,JL}, a first attempt to model the aggregation of {\sl Dyctiostelium discoideum}, the spatially homogeneous and dissipative Boltzmann equations for (simplified) granular flows \cite{BCCP,BCP,CMcCV,Villani.OT}, the Smoluchowski-Poisson equation modeling gravitational collapse \cite{B95,BN94,BilerWoyczynski}, and finally the McKean-Vlasov equations in kinetic theory \cite{Villani.OT}. In any of such mathematical model, the spatial dimension and the interaction potential kernel are chosen accordingly to the physical or biological phenomenon under consideration. However, the interaction potential kernel is always assumed to be even, thus reflecting the Newton's third law. Furthermore, they all share the property to be endowed with a {\sl free energy}, given by the combination of an internal energy with an interaction energy.  This common feature is due to the  attempt to model random but self-attractive movements of particles, whose mass density is conserved along time. More interestingly, the free energy gives the {\sl variational formulation} of the physical or biological phenomenon, since it is decreasing along the trajectory of the particles and satisfies an ``H-theorem''. As a consequences, the mathematical model can be seen (at least formally) as the {\sl gradient flow} of the energy in the
space of probability measures endowed with the appropriate metric. The latter fact being rigorously proved nowadays for some of the previously cited models~\cite{AGS,BCC}. 

Despite all this common features, the analytical behavior of the solutions of those type of equations can be completely different according to the choice of the interaction potential kernel. More precisely, if the interactions between particles do not enter in competition with the diffusion process (which would
be the case in \eqref{eq:gamma_eq}-\eqref{eq:Kgamma} with a negative $\gamma$), then the combined effect of the internal energy and the interaction energy gives rise to global solutions converging toward the equilibrium \cite{CMcCV,CarrilloToscani,Villani.OT}. On the other hand, if the interaction potential kernel is in competition with the diffusion process, as in \eqref{eq:gamma_eq}-\eqref{eq:Kgamma}, then the solutions can blow-up in finite time whenever the initial density is ``large'' and/or  ``concentrated'' enough. This is exactly the case of the parabolic-elliptic Keller-Segel system in dimension $d\ge2$ (\cite{BDP,CCE}) and of the Smoluchowski-Poisson equation (\cite{SireCh 04,SireCh 08}) cited above. This kind of scenario is much less understood than the previous one, for which very accurate results exist (see the references in \cite{Villani.OT}). This is also the reason why we consider here the initial valued problem associated to \eqref{eq:gamma_eq}-\eqref{eq:Kgamma}.

The family of equations \eqref{eq:gamma_eq}-\eqref{eq:Kgamma} has no physical interpretation, to the best of our knowledge, but can be interpreted as the projection of the parabolic-elliptic Keller-Segel system from the high dimension $d\ge3$ to the one spatial dimension. Indeed, the associated free energy is 
\begin{equation} 
\mathcal{E}_\gamma[\rho](t) =  \int_{\R} \rho(x,t)\log\rho(x,t)\, dx - \dfrac{1}{2} \int_{\R} \rho(x,t)\, (K_\gamma* \rho(t))(x)\, dx\,,
\label{eq:energy}
\end{equation}
where the first term is the internal energy $\cau(\rho)$, also called {\sl entropy} hereafter, and the second term is the {\sl interaction energy} ${\caw_\gamma}(\rho)$. Under regularity conditions it satisfies
\[ 
\dfrac d{dt} \mathcal{E}_\gamma[\rho](t) = - \int_\R \rho(x,t) \left|\partial_x\left(\log \rho(x,t) - (K_\gamma * \rho(t))(x)\right)\right|^2\, dx\, . 
\]
Therefore, thanks to the decreasing behavior of the energy functional and exploiting efficiently the competing convexities of the kernel $K_\gamma$ and of the internal energy density $\rho\log\rho$, it is proved that the initial valued problem associated to \eqref{eq:gamma_eq}-\eqref{eq:Kgamma} shows at least two blow-up criteria (see Proposition \ref{pr:BU continuous}), exactly as the high-dimensional parabolic-elliptic Keller-Segel system \cite{B95,CCE}. 

However, the goal of this paper is not to analyzed the continuous model \eqref{eq:gamma_eq}-\eqref{eq:Kgamma}, but a consistent approximation of its Lagrangian formulation. For that, taking advantage of the one dimensional setting, of the {\em non-negativity} of the solution $\rho$ corresponding to a non-negative initial  density  $\rho_0$ and of the {\em conservation of the initial mass}, {\sl i.e.}
\[
\int_{\R}\rho(x,t)\,dx=\int_{\R}\rho_0(x)\,dx=:M\,,\qquad t>0\,,
\]
the family of equations \eqref{eq:gamma_eq} can be reformulated in term of the pseudo inverse $X(\cdot,t)\,:\,(0,M)\to\R$ of the distribution function associated to $\rho$. For the pseudo-inverse function $X$, the energy functional \eqref{eq:energy} rewrites as 
\beq
\mathcal G_\gamma[X](t) =-\int_0^M\!\log \left(\partial_m X(m,t)\right)dm-\dfrac 1{2} \int_0^M\!\!\int_0^M K_\gamma(X(m,t)-X(m'\!,t))dmdm' ,
\label{eq:energy_X}
\eeq
while the family of equations \eqref{eq:gamma_eq}-\eqref{eq:Kgamma} rewrites as the integro-differential equations \eqref{eq:inverse_eq}. Moreover, the latter result to be the gradient flow of the energy functional \eqref{eq:energy_X} in the Hilbert space $L^2(0,M)$ (see Proposition \ref{prop:gradflow}). 

The particles approximation is then constructed from the mentioned gradient flow interpretation, thereby preserving all the properties of the continuous problem. It results in the dynamical system \eqref{eq:particle scheme} describing the trajectories $\cax(t)=(X_i(t))_{i=1,\dots,N}$ of $N\ge3$ particles distributed increasingly, {\sl i.e.}  lying in the set $\Pi=\{X\in\R^N\,:\,X_i<X_{i+1}\}$,  and carrying the same fraction $M(N+1)^{-1}$ of the conserved total mass $M$. For this dynamical system we give a global existence result (Theorem \ref{th:global existence}) based on a sufficient condition on the initial distribution of the particles that prevents collision. This sufficient condition is far from being optimal, but it is in some sense better than the sufficient condition for global existence of weak solutions of the parabolic-elliptic Keller-Segel system in high dimensions. Moreover, we obtain two blow-up criteria for the particles system (Propositions \ref{prop:discrete BU 1}  and \ref{prop:discrete continuum BU}) consistent with the blow-up criteria for the continuos equation (Proposition \ref{pr:BU continuous}) and incompatible with the global existence result. More specifically, the particles whose initial distribution satisfies one of the blow-up criteria and the particles whose initial distribution satisfies the global existence condition, lye in the subsets of $\Pi$ complementary w.r.t. the curve of the critical points of the discrete energy \eqref{eq:discrete_G}. Along this curve the time derivative of $|\cax(t)|$ is zero. However, the curve is not the unstable manifold separating the two basins of attraction (global existence and blow-up), as it is put in evidence numerically in the case of a three particles system with zero center of mass. In that case, all the mentioned criteria have been plotted in Figure \ref{fig:BU} and the unstable manifold computed numerically. The derivation of its equation would provide a single criterion to distinguish between global existence and blow-up. It would be also the proof that the dichotomy of the limit case $\gamma=0$ with logarithmic interaction kernel (see Section \ref{sec:gamma zero}) still holds true in the case $\gamma\in(0,1)$. The derivation of this single criterion is under investigation in a forthcoming paper. 

We conclude this introduction, observing that the Jensen inequality will play the main role in almost all the proofs. It is in fact quite the only tool used here. 
This reflects the fact that the competing convexities of the interaction potential kernel and of the internal energy density are the determinant keys of the dynamics of the particle system as well as of the continuous family of equations \eqref{eq:gamma_eq}.

This paper is organized as follows. In Section \ref{Sec:preliminaries} we summarize the analytical properties of problem \eqref{eq:gamma_eq}-\eqref{eq:Kgamma}. In Section \ref{Sec:continuous gradient flow} we reformulate the problem in term of the pseudo-inverse of the distribution function associated to $\rho$ and we give its gradient flow interpretation. In Section \ref{subsec:discrete} we introduce the $N$ particle scheme and we give a sufficient condition for the global existence of the trajectories. Section \ref{sec:discrete BU} is devoted to the analysis of the blow-up of the particles trajectories. Finally, in Section \ref{sec:gamma zero} we briefly recall the limit case $\gamma=0$ and we give the corresponding global existence result for the trajectories of $N\ge3$ particles. 
\section{Preliminaries}
\label{Sec:preliminaries}
In this section we briefly recall some facts and results about the existence and the blow-up of the solutions of the initial value problem for the family of equations \eqref{eq:gamma_eq}-\eqref{eq:Kgamma}. These results are inspired from those obtained in the analysis of the classical parabolic-elliptic Keller-Segel system. They are given here for completeness, but also to convince the reader that the re-writing of the family of equations \eqref{eq:gamma_eq} in term of the pseudo inverse of the distribution function of $\rho$ is the natural procedure to obtain explicit computations by its gradient flow interpretation, (see Section \ref{Sec:continuous gradient flow}).  

First of all, equation \eqref{eq:gamma_eq} is invariant under the space-time scaling $\rho_\lambda(x,t)=\lambda^{\gamma-1}\rho(\f x\lambda,\f t{\lambda^2})$, which in turn conserve the $L^{\f1{1-\gamma}}(\R)$ norm. This invariance property suggest that the Lebesgue space $L^{\f1{1-\gamma}}(\R)$ could be a critical space for the existence of global solutions of \eqref{eq:gamma_eq}-\eqref{eq:Kgamma}, as it is the case with $L^{\f d2}(\R^d)$ for the high dimensional Keller-Segel system \cite{CCE,CPZ}. However, due to the singularity of the drift term $\partial_xS_\gamma(\rho)$, inherited by  the strong singularity of the interaction potential kernel $K_\gamma$ in one dimension, this seems not to be the case for equation~\eqref{eq:gamma_eq}, (see also the end of Section \ref{subsec:discrete}). More specifically, the evolution equation of the $L^p$ norm of $\rho$ is given~by

\begin{equation}
\f d{dt}\int_\R\rho^p\,dx=-\f{4(p-1)}p\|\partial_x\rho^{\f p2}\|_{L^2(\R)}^2+2(p-1)\int_\R\rho^{\f p2}\partial_x\rho^{\f p2}(K_\gamma*\partial_x\rho)\,dx\,.
\label{eq:rho_p}
\end{equation}
From the weak Young inequality, the convolution with $K_\gamma$ is a bounded map from $L^m(\R)$ to $L^r(\R)$, with $1<m<r<\infty$ and $\f1r=\f1m+(\gamma-1)$. Therefore, for $1=\f1q+\f12+\f1m+(\gamma-1)$, $p\in(1,2]$ and $\f1m=\f12+\f{(2-p)}{pq}$, we have
\begin{equation}
\begin{split}
\int_\R&\rho^{\f p2}\partial_x\rho^{\f p2}(K_\gamma*\partial_x\rho)\,dx 
\le C\,\|\rho^{\f p2}\|_{L^q(\R)}\|\partial_x\rho^{\f p2}\|_{L^2(\R)}\|\partial_x\rho\|_{L^m(\R)}\\
&\le \f2p\,C\,\|\rho^{\f p2}\|_{L^q(\R)}\|\partial_x\rho^{\f p2}\|_{L^2(\R)}^2\|\rho^{1-\f p2}\|_{L^{\f{pq}{2-p}}(\R)}=\f2p\,C\,\|\rho\|_{L^{\f1{1-\gamma}}(\R)}\|\partial_x\rho^{\f p2}\|_{L^2(\R)}^2\,,
\end{split}
\label{est:rho_p}
\end{equation}
since, by the previous choice of indeces, $\f{pq}2=\f1{1-\gamma}$. Thus, plugging \eqref{est:rho_p} into \eqref{eq:rho_p}, it follows
\begin{equation}
\f d{dt}\int_\R\rho^p(x,t)\,dx\le\f{4(p-1)}p\|\partial_x\rho^{\f p2}(t)\|_{L^2(\R)}^2\left[C\,\|\rho(t)\|_{L^{\f1{1-\gamma}}(\R)}-1\right]
\,.
\label{eq:rho_p_finale}
\end{equation}
However, we can not deduce any control of $\int_\R\rho^{\f1{1-\gamma}}\,dx$ from the previous differential inequality, since we are not allowed to chose $p=\f1{1-\gamma}$ in \eqref{eq:rho_p_finale}. Indeed, on the one hand $1<p\le2$, and on the other hand, by the previous numerology it holds: $\f1r>0$, implying $q>2$, which in turn implies $p<\f1{1-\gamma}$. Consequently, it seems not possible to obtain a global existence result from a smallness condition on the $L^{\f1{1-\gamma}}$ norm of the initial density $\rho_0$.

The previous computation implies that a weak formulation of problem \eqref{eq:gamma_eq} should be considered. Owing to the property $K_\gamma(x)=K_\gamma(-x)$, it is natural to consider 
\begin{equation}
\begin{split}
\dfrac d{dt}\int_{\R}\rho(x,t)\, & \phi(x)\,dx= \int_{\R}\rho(x,t)\phi''(x)\,dx
\\
&+\f12\iint\limits_{\R\times\R}(\phi'(x)-\phi'(y))\rho(x,t)\f{(x-y)}{|x-y|^{\gamma+2}}\rho(y,t)\,dx\,dy\,,
\end{split}
\label{eq:weak_form}
\end{equation}
for any twice differentiable test function $\phi$ with bounded second derivative. 

Whenever $\rho(x,0)=\rho_0(x)$ in the sense of distribution, this formulation guarantees that the total mass $M$ is conserved. The existence of such a weak solution is an open question. A  possible way would be to regularize the problem, as done for instance in \cite{SenbaSuzuky}. An alternative proof of existence could be obtained using the Jordan-Kinderlehrer-Otto scheme \cite{AGS,JKO}, as in \cite{BCC}. However, we don't go through this matter since this is not the goal of the present paper. 

The weak formulation \eqref{eq:weak_form} allows also to obtain informations on ``moments'' of $\rho$. For instance, taking $\phi(x)=x$ in \eqref{eq:weak_form}, it follows that the center of mass of any weak solutions is conserved. Furthermore, for $\phi(x)=x^2$, it follows the evolution equation of the second moment $I(t):=\int_{\R}|x|^2\rho(x,t)\,dx$ of $\rho$, {\sl i.e.}
\beq
I'(t)=2M-\iint\limits_{\R\times\R}\rho(x,t)\,\f1{|x-y|^{\gamma}}\,\rho(y,t)\,dx\,dy\,.
\label{eq:I'(t)}
\eeq

When $\gamma=0$, equation \eqref{eq:I'(t)} gives the well known mass threshold phenomenon \cite{BDP,CPS} (see also Section \ref{sec:gamma zero}). In contrast with this simple dichotomy, in the case $\gamma\in(0,1)$ considered here, at least two blow-up criteria for the density $\rho$ can be deduced from \eqref{eq:I'(t)}. Indeed,  \eqref{eq:I'(t)} reads as
\beq
I'(t)=2M-\gamma\,{\caw_\gamma}(\rho)
=2M+2\,\gamma\,\mathcal{E}_\gamma[\rho]- 2\,\gamma\, \cau(\rho)\,.
\label{eq:I'(t)_gamma>0}
\eeq
It is then sufficient to obtain appropriate lower bounds for $\cau(\rho)$ or ${\caw_\gamma}(\rho)$ w.r.t. $I$ in order to get a differential inequality from \eqref{eq:I'(t)_gamma>0} assuring blow-up. 

\begin{lemma}
Let $\gamma\in(0,1)$ and $\rho\in L^1_+(\R;(1+|x|^2)\,dx)$. The following lower bounds for the interaction energy ${\caw_\gamma}(\rho)$ and the entropy $\cau(\rho)$ hold true
\beq
\gamma\,{\caw_\gamma}(\rho)\ge 2^{-\gamma/2}\,M^{2+\f\gamma2}\,I^{-\gamma/2}\,,
\label{eq:LB W_gamma}
\eeq
\beq
\cau(\rho)\ge-\f M2\log I+\f M2\log\left(\f {M^3}{2\pi e}\right).
\label{eq:LB U_gamma}
\eeq
\label{lm: continuous_LB}
\end{lemma}

\begin{proposition}[Continuous blow-up criteria]
Let $\gamma\in(0,1)$ and assume that the initial density $\rho_0\in L^1_+(\R;(1+|x|^2)\,dx)$ has finite energy and satisfies either
\beq
\int_{\R}|x|^2\rho_0(x)\,dx<\left(\f M2\right)^{\f2\gamma+1}\,,
\label{1st_BU}
\eeq
or
\beq
\int_{\R}|x|^2\rho_0(x)\,dx<\f{M^3}{2\pi\,e^{\f2\gamma+1}}\exp\left(-\f2M\,\mathcal{E}_\gamma[\rho_0]\right)\,.
\label{2nd_BU}
\eeq
Then, if the corresponding weak solution has finite energy, the solution blow-up in finite time.
\label{pr:BU continuous}
\end{proposition}

A proof of \eqref{eq:LB W_gamma} and \eqref{eq:LB U_gamma} can be found in \cite{B95} and \cite{CCE} respectively. However,  we shall show in Section \ref{Sec:continuous gradient flow} an alternative proof that make use of the change of variable introduced there. Let us just observe here that the l.h.s. of \eqref{eq:LB W_gamma} and \eqref{eq:LB U_gamma} might be $+\infty$. The proof of Proposition \ref{pr:BU continuous} is also standard and not given here. We only point out that both criteria \eqref{1st_BU} and \eqref{2nd_BU} are invariant under the space dilation $f_\lambda(x)=\lambda^{\gamma-1}f(\f x\lambda)$ induced by the space-time scaling previously introduced. The interested reader can found the blow-up criteria for the high-dimensional parabolic-elliptic Keller-Segel system equivalent to \eqref{1st_BU} and \eqref{2nd_BU} in \cite{CCE}. 
\section{The continuous equation as a gradient flow}
\label{Sec:continuous gradient flow}
We shall hereafter take advantage of the one dimensional setting of the problem to perform an explicit ``change of variable'' allowing us to rewrite the family of equations  \eqref{eq:gamma_eq} in a sort of Lagrangian formulation. For this purpose, we consider the distribution function associated to $\rho$, 
\[ 
F(x,t):=\int_{-\infty}^{x} \rho(y,t)\, dy\,,\qquad x\in\R\,,
\]
and its pseudo inverse $X(\cdot,t)\,:\,(0,M)\to\R$,
\[ 
X(m,t) := \inf\left\{ x\in\R \, \left|\, F(x,t)>m\right.\right\}. 
\]
Since $\rho(\cdot,t)\in L^1_+(\R)$, $F(\cdot,t)$ is an absolutely continuous non decreasing function and therefore  a.e. derivable over $\R$, while $X$ is a right continuous non decreasing function and a.e. derivable over $(0,M)$ with $\partial_m X\ge0$ (being $X$ the gradient of a convex function, see \cite{EG}). Moreover, the following identity holds true for any $m\in(0,M)$
\beq
F(X(m,t),t)=m\,,
\label{eq:F(X)}
\eeq
(since the measure $\rho\,dx$ do not charge points, otherwise $F(X(m,t),t)\ge m$).

Next, let us proceed with formal computations, without take into account the singularity of the drift term in \eqref{eq:gamma_eq}. Integrating the latter over $(-\infty,x)$, one obtains the equation for $F$ as
\[ 
\partial_tF=\partial_{xx}\,F - \partial_xF\, \partial_x S_\gamma(\partial_xF)\,.
\]
Using the differential relations obtained from the identity \eqref{eq:F(X)}, the previous equation reads as the following family of integro-differential equations
\begin{equation}
-\partial_t X(m,t)= \partial_m \left(\dfrac1{\partial_m X(m,t)} \right)+ 
\int_0^M\,\f{X(m,t) - X(m',t)}{|X(m,t) - X(m',t)|^{\gamma+2}}\, dm' \,.
\label{eq:inverse_eq} 
\end{equation}
Obviously, the singularity in \eqref{eq:gamma_eq} has been translated into \eqref{eq:inverse_eq}, giving rise to an integral term a priori not finite, and one should consider the following weak formulation 
\begin{equation}
\begin{split}
\dfrac d{dt}&\int_0^M\phi(X(m,t))\,dm=\int_0^M\phi''(X(m,t))\,dm
\\
&+\f12\int_0^M\!\!\!\int_0^M\left(\phi'(X(m,t))-\phi'(X(m',t))\right)K'_\gamma(X(m,t)-X(m',t))\,dm\,dm',
\end{split}
\label{eq:weak_form2}
\end{equation}
obtained taking into account the boundary conditions : 
\beq
\partial_m X(0,t)=\partial_m X(M,t)=+\infty\,,\quad X(0,t)=-\infty\quad \text{and}\quad X(M,t)=+\infty\,.
\label{eq:BC for X}
\eeq

The previous results can be made rigorous, recalling that the map $X$ transports $\caL_{[0,M]}$, the Lebesgue measure in the interval $[0,M]$, over the measure $\rho\,dx$, i.e. $X{\#}\caL_{[0,M]}$ has distribution function $F$. With this transport point of view in mind, it is easily seen that \eqref{eq:weak_form2} is the translation of \eqref{eq:weak_form} in the $m$ variable, that 
\beq
\int_0^MX(m,t)\,dm=\int_\R x\,\rho(x,t)\,dx\,, 
\label{eq:Xcentered}
\eeq
and $X(\cdot,t)\in L^2(0,M)$ has soon has $\rho$ has finite second moment, since
\beq
\int_0^MX^2(m,t)\,dm = \int_\R |x|^2\rho(x,t)\, dx\,.
\label{eq:I}
\eeq
Moreover, making use of the area formula, that can be applied here thanks to the regularity properties of the map $X$ (see \cite{AGS} page 130), the entropy $\cau(\rho)$ becomes 
\beq
\cau(X)=-\! \int_0^M\log \left(\partial_m X\right)dm\,,
\label{eq:entropy in X}
\eeq
and the energy \eqref{eq:energy} translates into \eqref{eq:energy_X}.  

Despite of the singularity of the integro-differential equation \eqref{eq:inverse_eq}, we deduce finally the gradient flow interpretation of \eqref{eq:inverse_eq}.
\begin{proposition}[Gradient flow interpretation] 
Let $\gamma\in(0,1)$. The integro-differential equation \eqref{eq:inverse_eq} is the gradient flow of the energy functional $\mathcal G_\gamma[X]$ for the Hilbertian structure over $L^2(0,M)$: 
\begin{equation}
\partial_t X = -\nabla_{L^2}\, \mathcal G_\gamma[X]\, . 
\label{eq:gradient_flow}
\end{equation}
\label{prop:gradflow}
\end{proposition}
\begin{proof} To prove \eqref{eq:gradient_flow}, we compute formally the first variation of the functional $\mathcal G_\gamma[X]$, omitting the $t$ variable for simplicity, to obtain
\[
\begin{split}
\delta\,\mathcal G_\gamma[X](Y)&=\f d{d\theta} \mathcal G_\gamma [X +\theta Y]_{|_{\theta=0}}=-\int_0^M\f1{\partial_mX(m)}\,\partial_mY(m)\,dm\\
&\quad-\f12\int_0^M\!\!\int_0^M K'_\gamma(X(m,t)-X(m'\!,t))(Y(m)-Y(m'))dmdm'\\
&=\int_0^M\partial_m\left(\f1{\partial_mX(m)}\right)Y(m)\,dm\\
&\quad-\!\int_0^M\!\!\!\int_0^M\! K'_\gamma(X(m,t)-X(m'\!,t))Y(m)\,dm\,dm'=-<\partial_tX,Y>.
\end{split}
\]
\end{proof}

\begin{remark} In order to obtain a rigorous proof of the previous proposition one should use the Jordan-Kinderlehrer-Otto scheme \cite{AGS,JKO}.
\end{remark} 

\begin{remark} [Scaling invariance] The space-time scaling leaving invariant the equation \eqref{eq:gamma_eq} translates into $X_\lambda(m,t)=\lambda\,X(\f m{\lambda^\gamma}, \f t{\lambda^2})$, the distribution function associated to $\rho_\lambda(x,t)=\lambda^{\gamma-1}\rho(\f x\lambda,\f t{\lambda^2})$ being $F_\lambda(x,t)=\lambda^{\gamma}F(\f x\lambda,\f t{\lambda^2})$. Furthermore, it is easily proved that the scaling $X\to X_\lambda$ leaves invariant the integro-differential equation \eqref{eq:inverse_eq}.
\label{rm:scaling}
\end{remark} 

We conclude this section observing that \eqref{eq:weak_form2} with $\phi(x)=x^2$ becomes \eqref{eq:I'(t)_gamma>0}, and giving the proof of Lemma \ref{lm: continuous_LB} with the help of the transport map $X$. 
\begin{proof}[Proof of Lemma \ref{lm: continuous_LB}] We shall consider densities $\rho$ with zero center of mass, without loss of generality. Inequality \eqref{eq:LB W_gamma} follows immediately by the Jensen inequality applied to the convex function $\phi(z)=z^{-\gamma/2}$, and \eqref{eq:Xcentered}, since
\[
\begin{split}
\gamma\,{\caw_\gamma}(\rho)&=\iint\limits_{\R\times\R}\rho(x)\,\f1{|x-y|^{\gamma}}\,\rho(y)\,dx\,dy=\int_0^M\!\!\!\int_0^M|X(m)-X(m')|^{-\gamma}dm\,dm'\\
&\ge M^2\,\left(M^{-2}\int_0^M\!\!\!\int_0^M|X(m)-X(m')|^2dm\,dm'\right)^{-\f\gamma2}
=(2I)^{-\f\gamma2}\,M^{2+\f\gamma2}.
\end{split}
\]

Next, let us assume $M=1$, again  without loss of generality. In terms of the pseudo inverse $X$, inequality \eqref{eq:LB U_gamma} reads as
\beq
-\! \int_0^1\log \left(X'(m)\right)dm+\f 12\log\left(\int_0^1|X(m)|^2dm\right)\ge \f 12\log\left(\f {1}{2\pi e}\right)\,.
\label{eq:LB_entropyinm}
\eeq
Let $\tilde\rho(x):=\f1{\sqrt{2\pi}}e^{-x^2/2}$, $G$ the distribution function associated to $\tilde\rho$ and $G^{-1}$ the inverse of $G$. Using $G$ as a change of variable, the identities $G''(x)=-x\,G'(x)$ and $\int_\R x^2\,G'(x)\,dx=1$, the Jensen inequality again and the Cauchy-Scwartz inequality, we have
\[
\begin{split}
-\! \int_0^1\log \left(X'(m)\right)dm&=-\! \int_\R\log \left((X(G(x)))'\right)G'(x)\,dx+\int_\R\log(G'(x))G'(x)\,dx\\
&\ge-\log\left(\int_\R(X(G(x)))'G'(x)\,dx\right)-\f12\log(2\pi e)\\
&=-\log\left(\int_0^1X(m)G^{-1}(m)\,dm\right)-\f12\log(2\pi e)\\
&\ge-\f12\log\left(\int_0^1|X(m)|^2\,dm\right)-\f12\log(2\pi e)\,,
\end{split}
\]
and \eqref{eq:LB_entropyinm} is proved.
\end{proof}

\begin{remark}
It is worth noticing that inequality \eqref{eq:LB_entropyinm} cannot be derived using a different probability measure (because of the identity $G''(x)=-x\,G'(x)$), but it is invariant w.r.t. any dilation $\tilde\rho_\lambda(x)=\f1\lambda\tilde\rho(\f x\lambda)$ of $\tilde\rho$. Moreover, it is an identity iff $X=G_\lambda^{-1}$, i.e. $\rho=\tilde\rho_\lambda$, with $\lambda=I^{\f12}$, since
\[
\int_0^1X(m)G^{-1}(m)\,dm=\f12\int_0^1|X(m)|^2\,dm+\f12-\f12\int_0^1|X(m)-G^{-1}(m)|^2dm\,.
\]
\end{remark}
\section{The $N$ particle scheme}
\label{subsec:discrete}
We are now ready to construct a finite dimensional dynamical system approximating the integro-differential equation \eqref{eq:inverse_eq} for the pseudo inverse $X$, adapted to the gradient flow structure \eqref{eq:gradient_flow}. The main point is to discretize the free energy $\mathcal G_\gamma[X]$ by using standard quadrature approximations and then to write the finite-dimensional gradient flow of the approximated energy. Following this procedure, we guarantee the preservation of some important properties at the discrete level, {\it e.g.} the dissipation of the energy, the homogeneities of the diffusion and interaction contributions, the conservation of the center of mass and the competing convexities. 

To begin with, we introduce the regular grid of points $m_i=i\,h$, $i=1,\dots,N$, over $(0,M)$, where $N$ is a fixed integer greater than two (the number of particles) and $h=M(N+1)^{-1}$ is the space step (chosen constant for simplicity). Then, we denote by $X_i(t)$ the approximation of $X(m_i,t)$ and $\cax(t)=(X_i(t))_{i=1,\dots,N}$. We make also the convention $X_0 = -\infty$ and $X_{N+1} = +\infty$. That is to say we put ghost particles at $\pm \infty$. This convention is consistent with the boundary conditions \eqref{eq:BC for X} and makes the computations below simpler.

Since we approximate a nondecreasing function $X(\cdot,t)$ and because the free energy becomes singular whenever particles collide, it make sense to look for {\em increasing} trajectories $\cax(t)$, {\sl i.e.} lying in the set $\Pi:=\{X\in\R^N\,:\,X_i<X_{i+1}\}$. Then, for the entropy $\cau(X)$ in \eqref{eq:entropy in X} we opt for the approximation
\beq
-\int_0^M\log \left(\partial_m X(m,t)\right)dm\ \Approx\ -\ h\sum_{i=1}^{N-1}\log\left(\f{X_{i+1}(t)-X_i(t)}h\right)=:U[\cax](t)\,,
\label{eq:entropy_approximation}
\eeq
while for the interaction energy $\caw(X)$ we opt for the approximation 
\beq
\f12\int_0^M\!\!\int_0^M K_\gamma(X(m,t)-X(m'\!,t))dm\,dm'\ \Approx\,
\f{h^2}2\sum_{i=1}^{N}\sum_{\substack{j=1\\ j\neq i}}^{N}K_\gamma(X_i(t)-X_j(t)):=W_\gamma[\cax](t)\,.
\label{eq:int_energy_approx}
\eeq
Therefore, we have derived as a natural approximation of $\mathcal G_\gamma[X]$, the following discrete energy functional
\beq
G_\gamma[X]:=U[X]-W_\gamma[X]\,,\quad X\in\Pi\,.
\label{eq:discrete_G}
\eeq
Accordingly, the finite dimensional dynamical system is defined as the gradient flow of $G_\gamma$ in the euclidian space $\R^N$
\beq
\dot{\cax}(t)=-\f1h\nabla G_\gamma[\cax](t)\,.
\label{eq:discrete_system}
\eeq

The dynamical system \eqref{eq:discrete_system} shares the same time/space invariance with the continuous setting (see Remark~\ref{rm:scaling}). It is in fact invariant under the rescaling $(h_\lambda,\cax_\lambda(t))=(\lambda^\gamma\,h,\lambda \cax(\f t{\lambda^2}))$. Rescaling the step $h$ is necessary in order to take into account that the mass $M$ scales into $\lambda^\gamma\,M$ under the dilation $f_\lambda(x)=\lambda^{\gamma-1}f(\f x\lambda)$. Moreover, system (4.4) preserves the center of mass, i.e. $\sum_{i=1}^N X_i(t)=\sum_{i=1}^N X_i(0)$. Indeed, using \eqref{eq:entropy_approximation}, \eqref{eq:int_energy_approx} and the previous convention, the dynamical system \eqref{eq:discrete_system} reads in details as follows
\beq 
\dot X_i(t) = - \f 1{X_{i+1}(t)-X_i(t)}+\f 1{X_i(t)-X_{i-1}(t)}+ 
h \sum_{\substack{j=1\\ j\neq i}}^{N} \frac{\sgn(X_j(t)-X_i(t))}{|X_j(t) - X_i(t)|^{\gamma+1}}\,,
\label{eq:particle scheme}
\eeq
for $i=1,\dots,N$. Then, by summing the differential equations \eqref{eq:particle scheme} over $i$, using the telescopic summation, the boundary conditions, and the symmetry of the interaction kernel, we obtain $\sum_i\dot X_i(t) = 0$.

The local in time existence of solutions $\cax(t)\in\Pi$ of the system of ODEs \eqref{eq:particle scheme} is a consequence of the Cauchy-Lipschitz Theorem. However, the maximal time of existence can be finite. Clearly, this happens when at least two particles collide, or equivalently as the trajectory $\cax(t)$ reaches the boundary of the set~$\Pi$. By the means of the function
\[ 
\phi_\gamma (X) =\f {h}{\gamma}\sum_{i = 1}^{N-1} \dfrac1{(X_{i+1} - X_i)^{\gamma}}\,,\quad X\in\Pi\,,
\]
we shall establish hereafter a sufficient condition on the size of the initial datum $\cax(0)$ that prevents collision of particles, and assure therefore the global existence of the trajectory $\cax(t)$. 

\begin{theorem}[Global existence] 
Let $\gamma\in(0,1)$ and $N\geq 3$. There exists a constant $c(N)>0$ such that the following condition on the initial datum $\cax(0)\in\Pi$, 
\beq
\gamma\,\phi_\gamma(\cax(0))< (1+(N-3)c(N))^{-1}\,,
\label{eq:smallnesscondition}
\eeq 
guarantees the global existence of the solution $\cax(t)$. 
\label{th:global existence}
\end{theorem}
\begin{proof}
We claim that $\phi_\gamma(\cax(t))$ is strictly decreasing in time if $\cax(t)$ is a trajectory starting from $\cax(0)$ satisfying \eqref{eq:smallnesscondition}. Since $\phi_\gamma(\cax(t))$ goes to $+\infty$ as  the trajectory $\cax(t)$ reaches the boundary of the set~$\Pi$, the claim gives us the proof. 

Using \eqref{eq:particle scheme}, we can decompose the evolution of $\phi_\gamma$ along the flow $\cax(t)$ into the diffusion contribution and the interaction contribution as following
\beq
\begin{split}
\f d{dt}\phi_\gamma ({\cal X}(t))&=-h\sum_{i = 1}^{N-1} \left( \dot X_{i+1}(t) - \dot X_{i}(t)\right) \dfrac1{(X_{i+1}(t) - X_i(t))^{\gamma + 1}}\\
&= h\sum_{i = 1}^N \dot X_{i}(t)\left[ \dfrac1{(X_{i+1}(t) - X_i(t))^{\gamma + 1}}  - \dfrac1{(X_{i}(t) - X_{i-1}(t))^{\gamma + 1}}  \right]\\
&=-I_\gamma(\cax(t))+J_\gamma(\cax(t))\,,
\end{split}
\label{eq:phi gamma}
\eeq
where
\[ 
I_\gamma(X):=h \sum_{i = 1}^N \left(\dfrac {1}{X_{i+1} - X_i} -\dfrac1{X_{i} - X_{i-1}} \right)\left(\f1{(X_{i+1} - X_i)^{\gamma + 1}}-\f1{(X_{i} - X_{i-1})^{\gamma + 1}}  \right)\,,
\]
is positive, and
\[ 
J_\gamma(X):=h^2\sum_{i = 1}^N \sum_{\substack{j=1\\ j\neq i}}^{N} \frac{\sgn(X_j-X_i)}{|X_j - X_i|^{\gamma+1}}
\left(\f1{(X_{i+1} - X_i)^{\gamma + 1}}-\f1{(X_{i} - X_{i-1})^{\gamma + 1}}  \right)\,.
\]
In turn, $J_\gamma(X)$ can be decomposed into the two interaction contributions due to contiguous and non-contiguous particles respectively, as following
\beq
\begin{split}
J_\gamma(X)&= h^{2} \sum_{i = 1}^N \left(\dfrac 1{(X_{i+1} - X_i)^{\gamma + 1}} - \dfrac 1{(X_{i} - X_{i-1})^{\gamma + 1}} \right)^2\\
&\quad+ h^{2}\sum_{i = 1}^N\sum_{\substack{j=1\\ |j-i|\ge2}}^{N}\frac{\sgn(X_j-X_i)}{|X_j - X_i|^{\gamma+1}} \left(\dfrac1{(X_{i+1} - X_i)^{\gamma + 1}}  - \dfrac1{(X_{i} - X_{i-1})^{\gamma + 1}}  \right)\\
&=:J^1_\gamma(X)+J^2_\gamma(X)\,.
\label{eq:Jgamma}
\end{split}
\eeq
Moreover, with the notations $Y_i:=X_{i+1}-X_i$, $i=0,\dots, N$, and $Y=(Y_i)_i$, for simplicity, we have the identity
\beq
\begin{split}
\gamma\, h^{-2}&\,\phi_\gamma(Y)\,I_\gamma(Y)=\left(\sum_{j=1}^{N-1}\f1{Y_j^\gamma}\right)
\left(\sum_{i=0}^{N-1}\left(\f1{Y_{i+1}}-\f1{Y_{i}}\right)\left(\f1{Y_{i+1}^{\gamma+1}}-\f1{Y_{i}^{\gamma+1}}\right)
\right)\\
&=\sum_{j=1}^{N-1}\f1{Y_j^{\gamma+1}}\left(\f1{Y_{j}^{\gamma+1}}-\f1{Y_{j-1}^{\gamma+1}}\right)
-\sum_{j=1}^{N-1}\f1{Y_j^{\gamma+1}}\left(\f1{Y_{j+1}^{\gamma+1}}-\f1{Y_{j}^{\gamma+1}}\right)+\tilde R(Y)\\
&=h^{-2}\,J^1_\gamma(Y)+\tilde R(Y)\,,
\end{split}
\label{eq:tilde R}
\eeq
with the positive remainder $\tilde R(Y)$
\[
\begin{split}
\tilde R(Y)&=\sum_{j=1}^{N-1}\sum_{\substack{i=0\\ i\neq j-1}}^{N-1}\f1{Y_j^\gamma}\f1{Y_{i+1}}
\left(\f1{Y_{i+1}^{\gamma+1}}-\f1{Y_{i}^{\gamma+1}}\right)
-\sum_{j=1}^{N-1}\sum_{\substack{i=0\\ i\neq j}}^{N-1}\f1{Y_j^\gamma}\f1{Y_{i}}
\left(\f1{Y_{i+1}^{\gamma+1}}-\f1{Y_{i}^{\gamma+1}}\right)\\
&=\sum_{j=1}^{N-1}\f1{Y_j^\gamma}\f1{Y_{j+1}}\left(\f1{Y_{j+1}^{\gamma+1}}-\f1{Y_{j}^{\gamma+1}}\right)
-\sum_{j=1}^{N-1}\f1{Y_j^\gamma}\f1{Y_{j-1}}\left(\f1{Y_{j}^{\gamma+1}}-\f1{Y_{j-1}^{\gamma+1}}\right)\\
&\quad+\sum_{j=1}^{N-1}\sum_{\substack{i=0\\ i\neq j-1,j}}^{N-1}\f1{Y_j^\gamma}\left(\f1{Y_{i+1}}-\f1{Y_i}\right)
\left(\f1{Y_{i+1}^{\gamma+1}}-\f1{Y_{i}^{\gamma+1}}\right)\\
&=\sum_{j=1}^{N-2}\f1{Y_j\,Y_{j+1}}\left(\f1{Y_{j+1}^{\gamma+1}}-\f1{Y_{j}^{\gamma+1}}\right)(Y_j^{1-\gamma}-Y_{j+1}^{1-\gamma})\\
&\quad+\sum_{j=1}^{N-1}\sum_{\substack{i=0\\ i\neq j-1,j}}^{N-1}\f1{Y_j^\gamma}\left(\f1{Y_{i+1}}-\f1{Y_i}\right)
\left(\f1{Y_{i+1}^{\gamma+1}}-\f1{Y_{i}^{\gamma+1}}\right)\,.
\end{split}
\]
Therefore, the evolution equation \eqref{eq:phi gamma} becomes
\beq
\f d{dt}\phi_\gamma ({\cal X}(t))=[\gamma\,\phi_\gamma ({\cal X}(t))-1]\,I_\gamma ({\cal X}(t))+J^2_\gamma ({\cal X}(t))-h^2\,\tilde R ({\cal X}(t))\,,
\label{eq: phi gamma 2}
\eeq
and it remains to estimate the remainder $J^2_\gamma(\cax(t))-h^2\,\tilde R(\cax(t))$. The difficulty here is to measure the relative differences $|X_j-X_i|$, where  $|j-i|\ge2$, w.r.t. $|X_{i\pm1}-X_i|$. For that reason we shall distinguish the cases of three and more particles.\\
{\sl The three particles case.} When $N=3$, the above reminder is negative since 
\[
\begin{split}
h^{-2}J^2_\gamma(Y)&-\tilde R(Y)=\f1{Y_2^{\gamma+1}}\f1{(Y_1+Y_2)^{\gamma+1}}+\f1{Y_1^{\gamma+1}}\f1{(Y_1+Y_2)^{\gamma+1}}\\
&-\f1{Y_1^\gamma}\f1{Y_2^{\gamma+2}}-\f1{Y_2^\gamma}\f1{Y_1^{\gamma+2}}
-\f1{Y_1\,Y_{2}}\left(\f1{Y_{2}^{\gamma+1}}-\f1{Y_{1}^{\gamma+1}}\right)(Y_1^{1-\gamma}-Y_{2}^{1-\gamma})\\
&\le-\f1{(Y_1\,Y_{2})^\gamma}\left(\f1{Y_1}-\f1{Y_2}\right)^2
-\f1{Y_1\,Y_{2}}\left(\f1{Y_{2}^{\gamma+1}}-\f1{Y_{1}^{\gamma+1}}\right)(Y_1^{1-\gamma}-Y_{2}^{1-\gamma})\,.
\end{split}
\]
The theorem is then easily proved in that case.\\
{\sl The high number of particles case.} When $N>3$, we shall take advantage of the positivity of $\tilde R(Y)$ and estimate $J^2_\gamma(Y)$ w.r.t. $J^1_\gamma(Y)$. Indeed, using the increasing behavior of the trajectory, telescopic summations and the boundary
conditions $Y_0=-\infty$ and $Y_N = +\infty$, from \eqref{eq:Jgamma} it follows
\beq
\begin{split}
h^{-2}J^2_\gamma(Y)&\le\sum_{i=3}^{N}\sum_{j=1}^{i-2}\f1{Y_{i-1}^{\gamma+1}}\left|\f1{Y_{i}^{\gamma+1}}-\f1{Y_{i-1}^{\gamma+1}}\right|
+\sum_{i=1}^{N-2}\sum_{j=i+2}^{N}\f1{Y_{i}^{\gamma+1}}\left|\f1{Y_{i}^{\gamma+1}}-\f1{Y_{i-1}^{\gamma+1}}\right|\\
&\le(N-2)\sum_{i=3}^{N}\sum_{j=i}^{N}\left(\f1{Y_{j-1}^{\gamma+1}}-\f1{Y_{j}^{\gamma+1}}\right)
\left|\f1{Y_{i}^{\gamma+1}}-\f1{Y_{i-1}^{\gamma+1}}\right|\\
&\quad+(N-2)\sum_{i=1}^{N-2}\sum_{j=1}^{i}\left(\f1{Y_{j}^{\gamma+1}}-\f1{Y_{j-1}^{\gamma+1}}\right)
\left|\f1{Y_{i}^{\gamma+1}}-\f1{Y_{i-1}^{\gamma+1}}\right|\\
&\le(N-2)\sum_{i=3}^{N}\left[\f12\,J^1_\gamma(Y)+\f12(N-i+1)\left(\f1{Y_{i}^{\gamma+1}}-\f1{Y_{i-1}^{\gamma+1}}\right)^2\right]\\
&\quad+(N-2)\sum_{i=1}^{N-2}\left[\f12\,J^1_\gamma(Y)+\f i2\left(\f1{Y_{i}^{\gamma+1}}-\f1{Y_{i-1}^{\gamma+1}}\right)^2\right]\\
&\le 2(N-2)^2h^{-2}J^1_\gamma(Y)\,.
\end{split}
\label{est: J2 vs J1}
\eeq
Then, plugging \eqref{est: J2 vs J1} into the differential equation \eqref{eq: phi gamma 2} and using the identity \eqref{eq:tilde R}, we obtain the differential inequality
\[
\f d{dt}\phi_\gamma ({\cal X}(t))\le[(1+2(N-2)^2)\gamma\,\phi_\gamma ({\cal X}(t))-1]\,I_\gamma ({\cal X}(t))\,,
\]
and the theorem follows again. 
\end{proof}

We conclude this section with some remarks on the smallness condition \eqref{eq:smallnesscondition} for global existence. The function $\phi_\gamma$ has been chosen for its homogeneity property, that allows us to obtain the differential inequality on $\phi_\gamma$ itself. In addition, $\phi_\gamma$ is invariant under the rescaling $(h_\lambda,X_\lambda)=(\lambda^\gamma\,h,\lambda\,X)$ of the discrete setting. More interestingly, $\phi_\gamma(X)$ controls the interaction potential $W_\gamma(X)$, since the following inequalities (also invariant) hold true for all $X\in\Pi$
\beq
h\,\phi_\gamma(X)<W_\gamma[X]< h\,\f N2\,\phi_\gamma(X)\,.
\label{eq:phi control W}
\eeq 
Therefore, it is natural to obtain a global existence criterion from the analysis of the evolution of $\phi_\gamma(\cax(t))$. Indeed, the discrete blow-up criteria in Section~\ref{sec:discrete BU} are based, roughly speaking, on the fact that {\sl if the interaction energy $W_\gamma$ is initially large, then it remains large and aggregation dominates diffusion}. 

The left inequality in \eqref{eq:phi control W} is an immediate consequence of the definition of $W_\gamma$. For the right one, we have
\beq
\begin{split}
2\gamma\,W_\gamma[X]&=h^2\sum_{i=1}^{N-1}\sum_{j=i+1}^{N}\f1{(X_j-X_i)^\gamma}+h^2\sum_{j=2}^{N}\sum_{i=1}^{j-1}\f1{(X_j-X_i)^\gamma}\\
&< h^2\sum_{i=1}^{N-1}(N-i)\f1{(X_{i+1}-X_i)^\gamma}+h^2\sum_{j=2}^{N}(j-1)\f1{(X_j-X_{j-1})^\gamma}\\
&=h^2\,N\sum_{i=1}^{N-1}\f1{(X_{i+1}-X_i)^\gamma}\,.
\end{split}
\label{eq:discrete HLS}
\eeq
Moreover, following the previous discretization procedure, it is easily seen that $\gamma\,h^\gamma\,\phi_\gamma(X)$ is an approximation of $\|\rho\|_{L^{\gamma+1}(\R)}^{\gamma+1}$. Taking this into account, the right inequality in \eqref{eq:phi control W} results a (non-optimal) discretisation of the following inequality in the continuous setting (derived from  the Hardy-Littlewood-Sobolev inequality plus interpolation)
\[
\int_\R\int_\R \rho(x)\f1{|x-y|^\gamma}\rho(y)\,dx\,dy\le C_\gamma\,M^{1-\gamma}\|\rho\|_{L^{\gamma+1}(\R)}^{\gamma+1}\,.
\]
The non optimality is due to the crude inequality established in \eqref{eq:discrete HLS}. Since the $L^{\gamma+1}$ norm is controlled by the presumably  critical norm $L^{\f1{1-\gamma}}$, {\sl i.e.} $\|\rho\|_{L^{\gamma+1}(\R)}^{\gamma+1}\le M^{\gamma}\,\|\rho\|_{L^{\f1{1-\gamma}}(\R)}$, the smallness condition \eqref{eq:smallnesscondition} on  $\phi_\gamma(\cax(0))$ is less restrictive than a smallness condition on the discretisation of the $L^{\f1{1-\gamma}}$ norm of $\rho_0$ (see Section~\ref{Sec:preliminaries}). 

\begin{remark}
Whenever an implicit Euler scheme w.r.t. the time variable $t$ is applied to the discrete dynamical system \eqref{eq:discrete_system}, one follow down to the space discrete Jordan-Kinderlehrer-Otto scheme \cite{AGS,JKO}, analyzed in \cite{BCC} in the limit case $\gamma=0$.
\end{remark}
\section{Blow-up criteria for the particle system}
\label{sec:discrete BU}
The main idea to derive blow-up criteria for the dynamical system~\eqref{eq:discrete_system} is to follow carefully the time evolution of the square of the euclidean norm $|\cax(t)|$, giving the approximation of the second moment of $\rho$ by~\eqref{eq:I}. Due to the gradient structure \eqref{eq:discrete_system}, it holds
\[
\f12\f d{dt}|\cax(t)|^2=-\f1h\,\f\partial{\partial\lambda}\left(G_\gamma[\lambda\cax](t)\right)_{|_{\lambda=1}}\,.
\]
Moreover, since
\[
G_\gamma[\lambda X]=-h(N-1)\log\lambda+U[X]-\lambda^{-\gamma}\,W_\gamma[X]\,,
\]
we have finally
\beq
\f h2\f d{dt}|\cax(t)|^2=h(N-1)-\gamma\,W_\gamma[\cax](t)=h(N-1)+\gamma G_\gamma[\cax](t)-\gamma U[\cax](t)\,.
\label{eq:I'_discrete}
\eeq
It remains to find lower bounds  for $W_\gamma$ and $U$, to be plugged in the differential equation~\eqref{eq:I'_discrete}. To begin with, we denote hereafter $A$ the $(N-1)\times N$ matrix of the system $Y_i=X_{i+1}-X_i$, $i=1,\dots, N-1$, $B$ the $N\times N$ matrix of the system obtained adding to the previous system the equation $\sum_{i=1}^NX_i=0$, and we establish the following technical Lemma, giving the euclidean norm $|X|$ in term of the relative differences $(X_{j}-X_i)$, $j>i$.
\begin{lemma} For any $X\in\R^N$ lying in the hyperplane $\sum_{i=1}^NX_i=0$, it holds
\beq
N\,|X|^2=\sum_{i=1}^{N-1}\sum_{j=i+1}^{N}(X_j-X_i)^2\,.
\label{eq:identity 1}
\eeq
Moreover, if $X\in\Pi_0:=\{X\in\R^N\,:\,X_i<X_{i+1}\text{ and } \sum_{i=1}^NX_i=0\}$ and $Y=A\,X$, it holds
\beq
|X|^2=Y^T(A\,A^T)^{-1}Y\,,
\label{eq:identity 2}
\eeq 
and
\beq
|X|^2\ge\f2N\sum_{i=1}^{N-1}\sum_{j=i}^{N-1}(X_{i+1}-X_i)(X_{j+1}-X_j)\,,
\label{eq:inequality 3}
\eeq
\label{lm:technical}
\end{lemma}
\begin{proof} Observing that $N\,X_i^2=\sum_{j=1}^{N}(X_i-X_j)X_i$, we have
\begin{align*}
N\,|X|^2&=\sum_{i=2}^{N}\sum_{j=1}^{i-1}(X_i-X_j)X_i+\sum_{i=1}^{N-1}\sum_{j=i+1}^{N}(X_i-X_j)X_i\\
&=\sum_{i=1}^{N-1}\sum_{j=i+1}^{N}(X_j-X_i)X_j+\sum_{i=1}^{N-1}\sum_{j=i+1}^{N}(X_i-X_j)X_i\,,
\end{align*}
and the identity \eqref{eq:identity 1} follows. Next, with $\tilde Y=(Y,0)=(A\,X,0)\in\R^N$, we have
\[
|X|^2=|B^{-1}\tilde Y|^2=\tilde Y^T(B\,B^T)^{-1}\tilde Y=Y^T(A\,A^T)^{-1}Y\,,
\]
and \eqref{eq:identity 2} is proved. Finally, observing that the matrix $(A\,A^T)^{-1}$ is symmetric with positive entries each of which, outside the main diagonal is greater or equal to $\f1N$, and on the main diagonal greater or equal $\f2N$, we obtain inequality~\eqref{eq:inequality 3}.
\end{proof} 
\begin{proposition}[Blow-up criterion induced by $W_\gamma$]
Let $\gamma\in(0,1)$ and $N\ge~3$. Assume that the initial data $\cax(0)\in\Pi_0$ satisfies 
\beq
|\cax(0)|^2< \f1h\,\left(\f M2\right)^{\f2\gamma+1}\dfrac {N^{2/\gamma}(N-1)}{(N+1)^{\f2\gamma+1} }\,.
\label{eq:criterion BU discrt 1}
\eeq
Then, the corresponding solution $\cax(t)\in\Pi_0$ vanishes in finite time. Moreover, criterion \eqref{eq:criterion BU discrt 1} is incompatible with condition \eqref{eq:smallnesscondition}. 
\label{prop:discrete BU 1} 
\end{proposition}
\begin{proof} From the finite form of the Jensen inequality applied to the convex function $\phi(z)=z^{-\gamma/2}$, we obtain for any $X\in\Pi$
\beq
\bea
\displaystyle
W_\gamma[X]&=\f{h^2}\gamma\,\f1{\mu(N)}\sum_{i=1}^{N-1}\sum_{j=i+1}^{N}\mu(N)\,\phi(|X_i-X_j|^2)
\\ \\ \displaystyle
&\ge\f{h^2}\gamma\,\f1{\mu(N)}\phi\left(\sum_{i=1}^{N-1}\sum_{j=i+1}^{N}\mu(N)\,|X_i-X_j)|^2\right)\,,
\eea
\label{est1:W_gamma}
\eeq
where $\mu(N)^{-1}=\f{N(N-1)}2$. Plugging~\eqref{eq:identity 1} into \eqref{est1:W_gamma}, we obtain the lower bound
\beq
W_\gamma[X]\ge\f{h^2}\gamma\,\f1{\mu(N)}\left(\f2{N-1}\,|X|^2\right)^{-\gamma/2}
\label{eq:lowerboundWgamma}
\eeq
and the differential inequality for $|\cax(t)|^2$
\[
\f d{dt}|\cax(t)|^2\le2(N-1)-h\,N(N-1)\,\left(\f{N-1}2\right)^{\gamma/2}\,|\cax(t)|^{-\gamma}\,.
\]
Criterion \eqref{eq:criterion BU discrt 1} follows immediately. The incompatibility with condition \eqref{eq:smallnesscondition}, is a direct consequence of the right inequality in \eqref{eq:phi control W} and of the lower bound for the discrete interaction energy
\[
W_\gamma[X]>\f h\gamma\,(N-1)\,,
\]
obtained plugging \eqref{eq:criterion BU discrt 1} into \eqref{eq:lowerboundWgamma}.
\end{proof}

It is worth noticing that the lower bound \eqref{eq:lowerboundWgamma} is the approximation of \eqref{eq:LB W_gamma} and that the former has been obtained exactly as the latter, i.e. using the Jensen inequality applied to the same convex function. Moreover,  
the discrete criterion \eqref{eq:criterion BU discrt 1} converges toward the continuous criterion \eqref{1st_BU} as $h\to0$, if the initial data $\cax(0)\in\Pi_0$ is chosen such that $h\,|\cax(0)|^2\to\int_\R|x|^2\rho_0(x)\,dx$ as $h\to0$.\\

It is much more tricky to derive lower bounds for the discrete entropy \eqref{eq:entropy_approximation} giving the corresponding blow-up criteria. The reason is that one has to use the finite form of the Jensen inequality in a sharp way, as it has been done in \eqref{est1:W_gamma}. Let us first deduce an intermediate criterion to illustrate this technical problem. 
\begin{proposition}[Blow-up criterion induced by $U$]
Let $\gamma\in(0,1)$ and $N\ge~3$. Assume that the initial data $\cax(0)\in\Pi_0$ satisfies 
\beq
|\cax(0)|^2< h^2(N-1)\,e^{-2/\gamma}\exp\left(-\f2{h(N-1)}\,G_\gamma[\cax](0)\right)\,.
\label{eq:criterion BU discrt 2}
\eeq
Then, the corresponding solution $\cax(t)\in\Pi_0$ vanishes in finite time. Moreover, criterion \eqref{eq:criterion BU discrt 2} is incompatible with condition \eqref{eq:smallnesscondition}. 
\label{prop:discrete BU 2} 
\end{proposition}
\begin{proof} Given $X\in\Pi_0$, from \eqref{eq:inequality 3} and the finite form of the Jensen inequality applied to the concave function $\log z$ we derive
\[
\begin{split}
\log|X|^2&\ge\log\left(\f2{N\mu(N)}\right)+\log\left(\mu(N)\sum_{i=1}^{N-1}\sum_{j=i}^{N-1}(X_{i+1}-X_i)(X_{j+1}-X_j)\right)
\\ 
&\ge\log(N-1)+\mu(N)\sum_{i=1}^{N-1}\sum_{j=i}^{N-1}\log((X_{i+1}-X_i)(X_{j+1}-X_j))
\\ 
&=\log(N-1)+N\mu(N)\sum_{i=1}^{N-1}\log(X_{i+1}-X_i)\,.
\end{split}
\]
Hence, being $U[X]=h(N-1)\log h-h\sum_{i=1}^{N-1}\log(X_{i+1}-X_i)$, it holds
\beq
U[X]\ge h(N-1)\log h-h\f{(N-1)}2\log|X|^2+h\f{(N-1)}2\log(N-1)\,.
\label{est:U_gamma}
\eeq
Plugging the lower bound \eqref{est:U_gamma} in \eqref{eq:I'_discrete} and using the decreasing behaviour of $G_\gamma[\cax](t)$, we get the differential inequality
\[
\begin{split}
\f d{dt}|\cax(t)|^2\le&\,2\,(N-1)+2\gamma\,h^{-1}G_\gamma[\cax](0)+\gamma(N-1)\log|\cax(t)|^2\\
&\quad-\gamma(N-1)\log(h^2(N-1))\,.
\end{split}
\]
Criterion \eqref{eq:criterion BU discrt 2} follows immediately. 

To prove that \eqref{eq:criterion BU discrt 2} is not compatible with condition \eqref{eq:smallnesscondition}, it suffices to write \eqref{eq:criterion BU discrt 2} equivalently as
\[
\f{h(N-1)}2\log|\cax(0)|^2+U[\cax](0)-\f{h(N-1)}2\log(h^2(N-1))<W_\gamma[\cax](0)-\f h\gamma(N-1)
\]
and to observe that the quantity in the l.h.s. of the above inequality is non-negative by \eqref{est:U_gamma}. The claim follows then by using the right inequality in \eqref{eq:phi control W}. 
\end{proof}
It is immediately seen that criterion \eqref{eq:criterion BU discrt 2} doesn't converge toward \eqref{2nd_BU} as $h\to0$. This is essentially because the lower bound \eqref{est:U_gamma} has been obtained using the non-sharp inequality \eqref{eq:inequality 3} (\eqref{eq:inequality 3} becomes an identity iff $N=3$). As a consequence, \eqref{est:U_gamma} is not consistent with \eqref{eq:LB_entropyinm}. However, \eqref{eq:criterion BU discrt 2} is contained in a continuum of blow-up criteria obtained mimicking the proof giving \eqref{eq:LB_entropyinm}. 

\begin{proposition}[A continuum of blow-up criteria induced by $U$]
Let $\gamma~\in~(0,1)$ and $N\ge~3$. Assume that the initial data $\cax(0)\in\Pi_0$ satisfies 
\beq
|\cax(0)|^2< h^2(N-1)^2\,C(N)\,e^{-\f2\gamma}\,\exp\left(-\f2{h(N-1)}\,G_\gamma[\cax](0)\right)\, ,  
\label{eq:criterion BU discrt 3}
\eeq
where $C(N):=\sup C(\mu)$, the supremum being taken over all $\mu=(\mu_0,\dots,\mu_N)$, with $\mu_i>0$, $i=1,\dots,N-1$, arbitrarily fixed, $\mu_0=\mu_N=0$, and
\beq
C(\mu)=\left(\sum_{i=1}^N(\mu_i-\mu_{i-1})^2\right)^{-1}\,\exp\left(\f2{N-1}\sum_{i=1}^{N-1}\log\mu_i\right)\,.
\label{eq:C mu}
\eeq
Then, the corresponding solution $\cax(t)\in\Pi_0$ vanishes in finite time. Moreover, criterion \eqref{eq:criterion BU discrt 3} is incompatible with condition \eqref{eq:smallnesscondition}. 
\label{prop:discrete continuum BU} 
\end{proposition}
\begin{proof} 
With $\mu$ arbitrarily given as above and for $X\in\Pi$, we have
\beq
\begin{split}
U[X]&=h(N-1)\log h-h\sum_{i=1}^{N-1}\log((X_{i+1}-X_i)\mu_i)+h\sum_{i=1}^{N-1}\log\mu_i\\ 
&\ge h(N-1)\log h-h(N-1)\log\left(\sum_{i=1}^{N-1}\f{\mu_i}{N-1}(X_{i+1}-X_i)\right)+h\sum_{i=1}^{N-1}\log\mu_i\\
&= h(N-1)\log(h(N-1))-h(N-1)\log\left(\sum_{i=1}^{N}X_i(\mu_{i-1}-\mu_i)\right)+h\sum_{i=1}^{N-1}\log\mu_i\\
&\ge h(N-1)\log(h(N-1))-\f{h(N-1)}2\log(|X|^2)+\f{h(N-1)}2\log C(\mu)\,.
\end{split}
\label{est:U_gamma 2}
\eeq
It remains to prove that the positive constant $C(\mu)$ is upper bounded, so that $C(N)$ is finite. Indeed, setting $\overline\mu=(\mu_1,\dots,\mu_{N-1})^T$, from the one hand we have
\[
\sum_{i=1}^N(\mu_i-\mu_{i-1})^2=\overline\mu^T(A\,A^T)\,\overline\mu\ge\lambda_1\sum_{i=1}^{N-1}\mu_i^2\,,
\]
where $\lambda_1=4\sin^2(\f\pi{2N})$ is the smallest eigenvalue of the squared matrix $A\,A^T$, and from the other hand 
\[
\exp\left(\f2{N-1}\sum_{i=1}^{N-1}\log\mu_i\right)=\prod_{i=1}^{N-1}\mu_i^{\f2{N-1}}\le \f1{N-1}\sum_{i=1}^{N-1}\mu_i^2\,,
\]
by the Jensen inequality. Hence, $C(\mu)\le({\lambda_1}(N-1))^{-1}$. We conclude plugging the lower bound  for $U[X]$ obtained optimizing \eqref{est:U_gamma 2} w.r.t. $\mu$, into \eqref{eq:I'_discrete} and using again the decreasing behaviour of $G_\gamma[\cax](t)$, so that
\[
\begin{split}
\f d{dt}|\cax(t)|^2\le&\,2\,(N-1)+2\gamma\,h^{-1}G_\gamma[\cax](0)+\gamma(N-1)\log(|\cax(t)|^2)\\ 
&-2\gamma(N-1)\log(h(N-1))-\gamma(N-1)\log C(N)\,.
\end{split}
\]
Criterion \eqref{eq:criterion BU discrt 3} follows immediately. The incompatibility of \eqref{eq:criterion BU discrt 3} with condition \eqref{eq:smallnesscondition} is proved exactly as in Proposition \ref{prop:discrete BU 2}.
\end{proof}

It is actually possible to prove that the new lower bound \eqref{est:U_gamma 2} for the entropy $U[X]$ is consistent with \eqref{eq:LB_entropyinm} and that criterion \eqref{eq:criterion BU discrt 3} converges toward \eqref{2nd_BU} as $N$ goes to infinity, since the constant $C(N)$ in \eqref{eq:criterion BU discrt 3} satisfies
\beq
\lim_{N\to+\infty}N^{-1}\,C(N)=\f1{2\pi\,e}\,.
\label{eq:limite C(N)/N}
\eeq

\begin{corollary}
The limit \eqref{eq:limite C(N)/N} holds true. 
\end{corollary}
\begin{proof}
First we assume $M=1$, without loss of generality, and we prove that
\beq
\limsup_{N\to+\infty}\ \log(N^{-1}\,C(N))\le-\log(2\pi e)\,.
\label{eq:limsup}
\eeq
The latter is an immediate consequence of \eqref{est:U_gamma 2}. Indeed, noting $H$ the inverse of the distribution function $G$ of $\f1{\sqrt{2\pi}}e^{-x^2/2}$, for $X=(H(\f i{N+1}))_{i=1,\dots,N}\in\Pi$ and appropriate $\theta_i\in[0,1]$, we have
\[
\begin{split}
\log(N^{-1}\,C(N))&\le \f2{h(N-1)}U[X]-2\log(h(N-1))+\log(|X|^2)-\log N\\
&=-\f2{(N-1)}\sum_{i=1}^{N-1}\log\left(H'\left(\f{i+\theta_i}{N+1}\right)\right)-2\log(h(N-1))\\
&\ +\log\left(\f1{N+1}\sum_{i=1}^{N}H^2\left(\f i{N+1}\right)\right)+\log\left(\f{N+1}N\right)\\
&=-\log(2\pi)-\f1{(N-1)}\sum_{i=1}^{N-1}H^2\left(\f {i+\theta_i}{N+1}\right)-2\log(h(N-1))\\
&\ +\log\left(\f1{N+1}\sum_{i=1}^{N}H^2\left(\f i{N+1}\right)\right)+\log\left(\f{N+1}N\right)\,.
\end{split}
\]
Passing to the limit as $N$ goes to infinity, we obtain \eqref{eq:limsup}. The limits above can be easily justified using the increasing behavior of $H$ and its symmetry with respect to $m=\f12$, so that $\int_0^1H^2(m)dm=2\int_0^{\f12}H^2(m)dm=2\int_{\f12}^1H^2(m)dm$. For instance, if $N=2n+1$, we have 
\[
I_N:=\f1{N+1}\sum_{i=1}^{N}H^2\left(\f i{N+1}\right)=\f2{N+1}\sum_{i=1}^{n}H^2\left(\f i{N+1}\right)\,,
\]
and
\[
2\int_{\f1{N+1}}^{\f12}H^2(m)\,dm\le I_N\le 2\int_{\f12}^1H^2(m)dm\,.
\]

To conclude, it remains to prove that there exists $\mu=(0,\mu_1,\dots,\mu_{N-1},0)$, with $\mu_i>0$, $i=1,\dots,N-1$, such that $N^{-1}\,C(\mu)$ converges toward $(2\pi e)^{-1}$ as $N$ goes to infinity. So, we consider $\mu_i=G'(H(\f iN))$, $i=1,\dots,N-1$. Then, from the one hand we have, for $i=1,\dots,N-2$ and appropriate $\theta_i\in[0,1]$,
\[
\begin{split}
\mu_{i+1}-\mu_{i}=\f1N\,G''\left(H\left(\f{i+\theta_i}N\right)\right)H'\left(\f{i+\theta_i}N\right)=-\f1N\,H\left(\f{i+\theta_i}N\right)\,,
\end{split}
\]
so that
\beq
N\sum_{i=1}^{N}(\mu_i-\mu_{i-1})^2=\f1N\sum_{i=1}^{N-2}H^2(\f{i+\theta_i}N)+\f N{2\pi}\exp(-H^2(\f1N))+\f N{2\pi}\exp(-H^2(\f{N-1}N))\,.
\label{eq:C(N)1}
\eeq
From the other hand
\beq
\left(\prod_{i=1}^{N-1}\mu_i\right)^{\f2{N-1}}=\f1{2\pi}\exp\left(-\f1{N-1}\sum_{i=1}^{N-1}H^2\left(\f iN\right)\right)\,.
\label{eq:C(N)2}
\eeq
Therefore, using \eqref{eq:C(N)1} and \eqref{eq:C(N)2}, we obtain the limit and the proof. 
\end{proof}

Finally, we shall show that the non optimal family of criteria obtained from \eqref{eq:criterion BU discrt 3} when $C(N)$ is replaced by $C(\mu)$, contains criterion  \eqref{eq:criterion BU discrt 2}. The latter is verified obviously iff $C(\mu)=\f1{N-1}$. With the constant $C(\mu)$ in \eqref{eq:C mu} written equivalently as
\[
C(\mu)=\left(2\sum_{i=1}^{N-1}\mu_i(\mu_i-\mu_{i-1})\right)^{-1}\left(\prod_{i=1}^{N-1}\mu_i\right)^{\f2{N-1}}\,,
\]
and $\mu_1=\nu>0$, $\mu_i=1$, $i=2,\dots,N-1$, the identity $C(\mu)=\f1{N-1}$ becomes
\[
\nu^{\f2{N-1}}=\f2{N-1}(\nu^2-\nu+1)\,,
\]
and the above equation has at least one solution for any $N\ge3$, (exactly one solution for $N=3$ and two for $N>3$).

\begin{remark}
It is easily seen that the blow-up criteria \eqref{eq:criterion BU discrt 1}, \eqref{eq:criterion BU discrt 2} and \eqref{eq:criterion BU discrt 3} are all invariant w.r.t. the scaling $(h_\lambda,X_\lambda)=(\lambda^\gamma\,h,\lambda\,X)$ of the discrete setting. 
\end{remark}
\section{The discrete threshold phenomenon for $\gamma=0$}
\label{sec:gamma zero}
In the limit case $\gamma=0$, the interaction potential kernel $K_\gamma$ has to be replaced with $K_0(x)=-\log|x|$ in equation~\eqref{eq:gamma_eq}. This problem reproduces in any dimension $d\ge1$, the well know critical mass phenomenon of the two dimensional parabolic-elliptic Keller-Segel system \cite{BDP}, basically because both the internal energy $\cau(\rho)$ and the interaction energy $\caw_0(\rho)$ scale additively with a logarithmic correction, under the action of the mass-preserving dilation $f^\lambda(x)=\f1{\lambda^d}\,f(\f x\lambda)$, {\sl i.e.}
\[
\mathcal{E}_0[\rho^\lambda ] =M\,\left( \dfrac{M}2 - d\right)\log\lambda+\mathcal{E}_0[\rho]\,.
\]
Then, the energy is bounded from below on the set of integrable densities $\rho$ with finite energy iff $M\le M_c=2d$, (see also \eqref{eq:I'(t)}). 

This problem has been analyzed, for any dimension $d\ge1$, in the Eulerian formulation \eqref{eq:gamma_eq} in \cite{CPS}, and following the  gradient flow formulation \eqref{eq:gradient_flow}, written in self-similar variables, in \cite{BCC} (see also \cite{CalvezCarrillo}). More specifically, the authors in  \cite{BCC} proved the convergence of the Jordan-Kinderlehrer-Otto scheme \cite{JKO} adapted to that problem when $M<M_c$. However, without going through this generality, if $d=1$ it is still possible to define the particle scheme when $\gamma=0$, setting 
\beq
W_0[X]:=-h^2\sum_{i=1}^{N-1}\ \sum_{j=i+1}^{N}\log(X_j-X_i)
\quad\text{and}\quad
\dot{\cax}(t)=-\nabla G_0[\cax](t)\,.
\label{eq:gamma=0}
\eeq
Then, thanks to the homogeneity property of the corresponding discrete energy, i.e. 
\[
G_0[\lambda X]=-h(N-1)\log\lambda+\f{h^2}2\,N(N-1)\log\lambda+G_0[X]\,,
\]
it holds
\[
\f12\f d{dt}|\cax(t)|^2=h(N-1)\left(1-\f{hN}2\right)\,,
\]
and the solution blows up iff $h>\f2N$, i.e. iff
\beq
M>2\left(\f{N+1}N\right)\,.
\label{eq:BU gamma=0}
\eeq
It is worth noticing that the above discrete blow-up criterion converges as $N~\to~\infty$ toward the continuous blow-up criterion $M>2$. 

We conclude this short review of the case $\gamma=0$ with the proof of the global existence of a solution of the particle system under the condition reverse to \eqref{eq:BU gamma=0}, showing that the threshold phenomenon holds true also at the discrete level.  To the best of our knowledge, this is a new result. 

\begin{theorem}[Global existence]
Let $\gamma=0$ and $N\ge3$. If 
\beq
M<2\left(\f{N+1}N\right)\,,
\label{eq:global condition gamma 0}
\eeq
then for any $\cax(0)\in\Pi$ there exists a unique global solution $\cax(t)\in\Pi$ of the dynamical system \eqref{eq:gamma=0}. Moreover, if $N=3$, the discrete entropy $U[\cax](t)$ is decreasing in time. 
\end{theorem}
\begin{proof}
The proof of the local existence of a trajectory $\cax(t)$ starting from $\cax(0)$ is exactly the same as for $\gamma\in(0,1)$. To prove that this trajectory is global in time under condition \eqref{eq:global condition gamma 0}, {\sl i.e.} none of the contiguous particles collide, it is enough to prove that the discrete entropy $U[X]$ stay upper bounded along the trajectory $\cax(t)$ globally in time, since $U[X]=h\log\left(\prod_{i=1}^{N-1}h\,(X_{i+1}-X_i)^{-1}\right)$ goes to $+\infty$ as a relative difference $(X_{i+1}-X_i)$ vanishes. 

The proof is inspired by the theory of the two dimensional Keller-Segel system. Indeed it is based on the following (non-optimal) discretisation of the well known logarithmic Hardy-Littlewood-Sobolev inequality \cite{CarlenLoss}, 
\beq
2W_0[X]<h\,N\,U[X]-h^2N(N-1)\log h\,,\quad X\in\Pi\,,
\label{eq:discrete log HLS}
\eeq
that can be obtained exactly as the right inequality in \eqref{eq:phi control W}. Let us observe that the non-optimality concerns only the constant in the r.h.s. of~\eqref{eq:discrete log HLS} that goes to $+\infty$ as $h$ goes to 0, and not the constant factor of $U[X]$. That is the reason why we are nonetheless able to obtain the sharp global existence condition \eqref{eq:global condition gamma 0}. So, as in the continuous setting, using \eqref{eq:discrete log HLS} and the decreasing behavior of the discrete energy $G_0[X]$ along the trajectories, we have for $\theta>0$
\[
\begin{split}
(1-\theta)U[\cax](t)&\le G_0[\cax](0)-\theta U[\cax](t)+W_0[\cax](t)\\
&\le G_0[\cax](0)+\left(1-\f{2\theta}{h\,N}\right)W_0[\cax](t)-\theta\,h(N-1)\log h\,.
\end{split}
\]
Choosing $\theta=\f{h\,N}2$, we obtain a uniform in time upper bound for $U[\cax](t)$ under condition~\eqref{eq:global condition gamma 0}. 

Finally, if $N=3$, the decreasing behavior of $U[\cax](t)$ relies on the decreasing behavior of the function $\phi_0(X):=-h\sum_{i=1}^{N-1}\log(X_{i+1}-X_i)$, limit of the function $\phi_\gamma(X)$ as $\gamma$ goes to 0. Indeed, $\phi_0(\cax(t))$ evolves according to the equation
\[
\begin{split}
\f d{dt}\phi_0(\cax(t))&=2h^2(h-1)\left((X_2-X_1)^{-2}+(X_3-X_2)^{-2}-(X_2-X_1)^{-1}(X_3-X_2)^{-1}\right)\\
&\quad+h^3(X_2-X_1)^{-1}(X_3-X_2)^{-1}\,.
\end{split}
\]
Since $h<\f23$ by \eqref{eq:global condition gamma 0}, applying the Jensen inequality the above equation becomes  
\[
\f d{dt}\phi_0(\cax(t))<h^2(3h-2)(X_2-X_1)^{-1}(X_3-X_2)^{-1}
\]
and the claim follows. 
\end{proof}
\begin{figure}
\begin{center}
\includegraphics[width = \linewidth,height = .8\textheight]{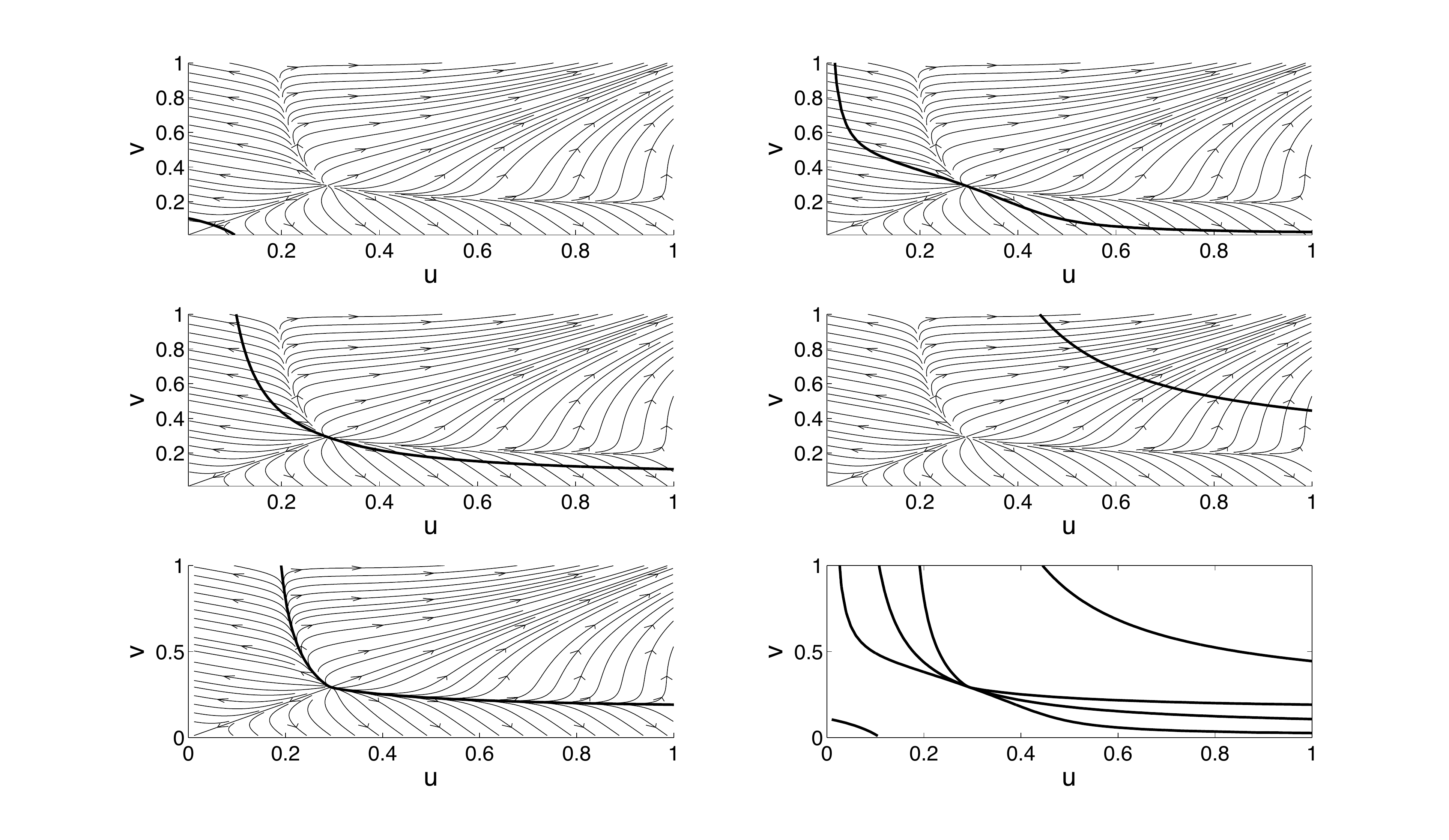}
\caption
{\small Behaviour of the three particles scheme, with $M=1$ and $\gamma=\f12$, in the phase plane $(u,v)=((X_2-X_1),(X_3-X_2))$. {\em Top left} : the blow-up criterion \eqref{eq:criterion BU discrt 1} is figured as a bold line. {\em Top right} : the  blow-up criterion \eqref{eq:criterion BU discrt 3} is figured as a bold line. {\em Middle left}: the curve $\gamma\,W_\gamma[X] = h(N-1)$ of the critical points of the energy, is plotted in bold. {\em Middle right} :  the global existence criterion \eqref{eq:smallnesscondition} is figured as a bold line. {\em Bottom left} : the unstable manifold starting from the energy maximal point and separating the two basins of attraction (global existence and blow-up) has been computed numerically and plotted in bold. {\em Bottom right} : all the previous lines are plotted on the same figure for the sake of comparison.
}
\label{fig:BU}
\end{center}
\end{figure}
\section*{Acknowledgements}
L.C. would like to thanks Francis Hirsch for helpful discussion. 
%
%


\noindent ${^a}$ Unit\'e de Math\'ematiques Pures et Appliqu\'ees,  CNRS UMR 5669 \\
\'Ecole Normale Sup\'erieure de Lyon, \\
46 all\'ee d'Italie, F~69364 Lyon cedex 07, France \\
e-mail: vincent.calvez@umpa.ens-lyon.fr
\bigskip

\noindent ${^b}$ Laboratoire d'Analyse et Probabilit\'e, \\
Universit\'e d'Evry Val d'Essonne, \\
23 Bd. de France, F--91037 Evry Cedex, France \\
e-mail: lucilla.corrias@univ-evry.fr

\end{document}